\newif\ifpartone
\partonetrue 

\makeatletter
\@namedef{subjclassname@2020}{%
  \textup{2020} Mathematics Subject Classification}
\makeatother

\documentclass[10pt]{amsart}

\usepackage{mathpazo}
\usepackage[mathpazo]{flexisym}
\usepackage{breqn}

\usepackage{verbatim}
\usepackage{graphicx}
\usepackage{xcolor}
\usepackage{amsfonts,amsmath,latexsym,amssymb,amsthm}
\usepackage{hyperref}

\newcounter{theoremcounter}

\newcounter{dummycounter}
\newcounter{propcounter}
\newcounter{corcounter}

\newcounter{emptycounter}

\newtheorem{theorem}[theoremcounter]{Theorem}

\newtheorem{proposition}[propcounter]{Proposition}
\newtheorem{corollary}[corcounter]{Corollary}

\newcounter{eqncounter}

\numberwithin{equation}{eqncounter}

\newcommand{\vecmahlermap}{\ensuremath{{\mathfrak{M}}}}
\newcommand{\vecmahler}[1]{\ensuremath{\vecmahlermap(#1)}}
\newcommand{\absvecdelta}[1]{\ensuremath{|\Delta_{#1}|}}

\newcommand{\vecconstant}[1]{\ensuremath{C(#1)}}

\newcommand{\vectorlist}{\ensuremath{\boldsymbol{\alpha}}}
\newcommand{\vectori}[1]{\ensuremath{{\boldsymbol{\alpha_{#1}}}}}
\newcommand{\altvectori}[1]{\ensuremath{{\boldsymbol{a_{#1}}}}}

\newcommand{\smallervectorsi}[1]{\ensuremath{\beta_{#1}}}
\newcommand{\smallervector}{\ensuremath{{\boldsymbol{\beta}}}}
\newcommand{\littleconstant}{\ensuremath{C_{0}}}
\newcommand{\radius}{\ensuremath{r}}

\newcommand{\altroottwoi}[1]{\ensuremath{a_{#1}}}
\newcommand{\mahlerfuncmap}{\ensuremath{M}}
\newcommand{\mahlerfunc}[1]{\mahlerfuncmap(#1)}

\newcommand{\polynomial}{\ensuremath{f}}
\newcommand{\rooti}[1]{\ensuremath{\alpha_{#1}}}
\newcommand{\mahdelta}[1]{\ensuremath{\Delta_{#1}}}

\def\IR{\mathbb R}
\def\IC{\mathbb C}
\def\IZ{\mathbb Z}

\def\IQ{\mathbb Q}

\def\Oseen{{\mathcal{O}}}

\title{On Mahler's inequality and small integral generators of totally complex number fields} 

\begin{document}

\author{Murray Child}
\author{Martin Widmer}

\address{Department of Mathematics\\ 
Royal Holloway, University of London\\ 
TW20 0EX Egham\\ 
UK}

\email{murray.child.2020@live.rhul.ac.uk}
\email{martin.widmer@rhul.ac.uk}

\date{\today}

\subjclass[2020]{Primary 11R06; 30C10   Secondary 11G50; 11R04}

\keywords{Mahler measure, Mahler's inequality, discrete logarithmic energy, small generators, number fields, Weil height}

\maketitle

\begin{abstract}
We improve Mahler's lower bound for the Mahler measure in terms of the discriminant and degree 
for a specific class of polynomials: complex monic polynomials of degree $d\geq 2$ such that all roots with modulus greater than some fixed value $r\geq1$ occur in equal modulus pairs. We improve Mahler's exponent  $\frac{1}{2d-2}$ on the discriminant to $\frac{1}{2d-3}$.
Moreover, we show that this value is sharp, even when restricting to minimal polynomials of integral generators of a fixed not totally real number field.

An immediate consequence of this new lower bound is an improved lower bound for integral generators of number fields,
generalising a simple observation of Ruppert from imaginary quadratic to  totally complex number fields of arbitrary degree.

\end{abstract}

\section{Introduction}

In this short note we prove a new lower bound for the Mahler measure of a monic polynomial in $\IC[x]$
whose ``large'' roots come in pairs of equal modulus. The bound is expressed in terms of the degree and the discriminant, and the dependence on the discriminant  is best-possible. The result implies a new lower bound for the smallest integral
generator of a totally complex number field. 

We refer the reader to Smyth's survey article  \cite{m_measure_survey} (see also \cite{roots_of_unity}) for a detailed account on the Mahler measure and its significance.

Our proofs are completely elementary and straightforward but the results seem to close a gap in the literature.

Let $f=a_0(x-\alpha_1)\cdots(x-\alpha_d)\in \IC[x]$ be of degree $d\geq 2$.
The \textit{Mahler measure} $\mahlerfunc{f}$  and the \textit{discriminant} $\mahdelta{f}$ of $f$ are defined by
\begin{alignat*}1
    \mahlerfunc{f}:=|a_{0}|\prod_{i=1}^{d}\max\{{1,|\alpha_{i}|\}},\\
    \mahdelta{f}:= a_{0}^{2d-2}\prod_{1 \leq i < j \leq d}(\alpha_{i} - \alpha_{j})^{2}.
\end{alignat*}
Mahler's classical inequality bounds the former from below in terms of the latter.
\begin{theorem}[Mahler, 1964 \cite{mahler_1964}] \label{thm:mahleroriginal} 
Let $\polynomial \in \mathbb{C}[x]$ be a polynomial of degree $d \geq 2$. Then
\begin{equation*}
    \mahlerfunc{\polynomial}^{2d-2} \geq d^{-d} |\mahdelta{\polynomial}|.
\end{equation*}
\end{theorem}

Moreover, Mahler showed that we have equality if and only if $\polynomial = \altroottwoi{0}x^{d}+\altroottwoi{d}$ with $|\altroottwoi{0}|=|\altroottwoi{d}|>0$. Note that for irreducible $\polynomial \in \mathbb{Z}[x]$ this only happens when $\polynomial = \pm(x^{d}+1)$ and $d$ is a power of $2$. However, for primes $p$ 
 the irreducible polynomials $f = (p+1)x^d-p$ in $\mathbb{Z}[x]$ satisfy\footnote{A simple calculation gives $|\mahdelta{f}| = d^d\left(p(p+1)\right)^{d-1}$ and $\mahlerfunc{f} = p+1$.} 
$$\mahlerfunc{f}^{2d-2} = d^{-d}|\mahdelta{f}|\left(\frac{p+1}{p}\right)^{d-1},$$
and thus even when restricting to irreducible polynomials 
neither the exponent ${2d-2}$ nor the constant $d^{-d}$ can be improved.
What if we restrict to monic irreducible polynomials in $\IZ[x]$?
Using Swan's discriminant formula for trinomials \cite[Theorem 2]{Swan1962} we get,
for the $p$-Eisenstein polynomial $f=x^d+px^{d-1}+(-1)^{d+1}p$, 
$$|\mahdelta{f}|=
p^{d-1}\left(d^d+\left((d-1)p\right)^{d-1}\right)
> (d-1)^{d-1}p^{2d-2}.$$
On the other hand, $\mahlerfunc{f}\leq (2+1/p)p<3p$ (using inequality (\ref{ineq:L1norm})), and therefore
$$\mahlerfunc{f}^{2d-2}<\left(\frac{9}{d-1}\right)^{d-1}|\mahdelta{f}|.$$
Hence, at least the exponent $2d-2$ is sharp, even when restricting to monic  irreducible polynomials in $\IZ[x]$.

Next, let us consider a monic quadratic polynomial $\polynomial = x^2 + bx + c$ with coefficients $a,b \in \mathbb{R}$ and no real roots.
 Then, as noted by Ruppert \cite{ruppert},
\begin{equation} \label{quarter}
\mahlerfunc{\polynomial} = 
\max\left\{1, \frac{b^2+|\mahdelta{\polynomial}|}{4}\right\}\geq \frac{|\mahdelta{\polynomial}|}{4}.
\end{equation}
So instead of Mahler's exponent ${2d-2} = {2}$ one obtains the much better exponent $1$. This raises the question of whether one can improve Mahler's exponent ${2d-2}$ when restricting to monic $\polynomial \in \mathbb{R}[x]$ with no real roots - and in particular, what is the sharp exponent in this case? We have been unable to find the answer in the literature but our next result shows that for such polynomials the exponent can be improved to ${2d-3}$.

\begin{theorem} \label{thm:mainresultformahler}
Let $m$ be non-negative and even, let $r\geq1$, and let
$$\polynomial(x)=\left(\prod_{i=1}^{\frac{m}{2}}(x-\rooti{i})(x-\rooti{i}')\right)\left(\prod_{i=\frac{m}{2}+1}^{d-\frac{m}{2}}(x-\rooti{i})\right) \in \mathbb{C}[x]$$ 
be a monic polynomial of degree $d\geq2$, with $|\rooti{i}| = |\rooti{i}'|$ for all $i=1,\ldots,\frac{m}{2}$ and $|\rooti{i}| \leq r$ for all $i=\frac{m}{2}+1,\ldots,d-\frac{m}{2}$. Then
\begin{equation*}
    \mahlerfunc{\polynomial}^{2d-3} \geq (2r)^{d(1-d)}|\mahdelta{\polynomial}|.
\end{equation*}
\end{theorem}

The exponent $2d-3$  in Theorem \ref{thm:mainresultformahler} is sharp, even when restricting to minimal polynomials of integral generators of a given not totally real number field. 

Let $K$ be a number field of degree $d$ and $\Oseen_K$ its ring of integers. 
For $\alpha \in K$ we write $f_{\alpha,\mathbb{Z}}$ for the minimal polynomial\footnote{The minimal polynomial $f_{\alpha,\mathbb{Z}}$ of $\alpha$ over $\mathbb{Z}$ is the unique  polynomial in $\IZ[x]$ of minimal degree
with positive leading coefficient and coprime coefficients that vanishes at $\alpha$.} of $\alpha$ over $\mathbb{Z}$. Recall that $\alpha \in \Oseen_K$ if and only if $f_{\alpha,\mathbb{Z}}$ is monic.
\begin{proposition} \label{thm:kcomplementtomahlerresult}
Let $K$ be a number field of degree $d$ with a non-real embedding. Then there exists $c_K > 0$ depending only on $K$ such that there are infinitely many $\alpha \in \Oseen_K$ with $K = \mathbb{Q}(\alpha)$,
\begin{equation*}
    M(f_{\alpha,\mathbb{Z}})^{2d-3} \leq c_K|\Delta_{f_{\alpha,\mathbb{Z}}}|,
\end{equation*}
and all roots of $f_{\alpha,\mathbb{Z}}$, except one pair of complex conjugate roots, have modulus at most $c_K$.
\end{proposition}
Under the hypothesis that the number field $K$ possesses a real embedding one can prove in the same way that there are 
infinitely many $\alpha \in \Oseen_K$ with $K = \mathbb{Q}(\alpha)$ and $M(f_{\alpha,\mathbb{Z}})^{2d-2} \leq c_K|\Delta_{f_{\alpha,\mathbb{Z}}}|$,
and all roots of $f_{\alpha,\mathbb{Z}}$ except one  have modulus at most $c_K$.

Next we apply Theorem \ref{thm:mainresultformahler} to get a new lower bound for integral generators of totally complex number fields.

Let $K$ be a number field. Then every $\alpha\in K$ with $K=\mathbb{Q}(\alpha)$  is called a \textit{generator} of $K$. It is natural to ask: how ``large" must a generator of $K$ be in terms of the degree $d$ and the modulus of the discriminant $|\Delta_{K}|$? And what happens if we restrict to \textit{integral} generators? 
This problem  has been studied by several authors, including Cochrane et al. \cite{real_quadratic}, Dubickas \cite{dubickas}, Eldredge \& Petersen \cite{eldandpet},  Kihel \& Lizotte \cite{kihel_lizotte}, Pierce \& Turnage-Butterbaugh \& Wood \cite{pierce_tb_wood_classgroups}, Ruppert \cite{ruppert}, and Vaaler \& Widmer \cite{wout_smol_gens, a_note}.

A good measure for size here is the Mahler measure $M(f_{\alpha,\mathbb{Z}})$ of the minimal polynomial $f_{\alpha,\mathbb{Z}}$ of $\alpha$ over $\mathbb{Z}$. Following  Eldredge \& Petersen \cite{eldandpet} and Dubickas \cite{dubickas}  we use the following notation:
\begin{equation*}
M(K) := \text{min}\{M(f_{\alpha,\mathbb{Z}}): \alpha \in K, \mathbb{Q}(\alpha) = K\},
\end{equation*}
and
\begin{equation*}
M(\Oseen_K) := \text{min}\{M(f_{\alpha,\mathbb{Z}}): \alpha \in \Oseen_K, \mathbb{Q}(\alpha) = K\}.
\end{equation*}
Writing $\|f\|_{1}$ for the $L^1$-norm of the coefficient vector of $f$
one has by a result of Mahler \cite{mahler_1960}
\begin{equation}\label{ineq:L1norm}
2^{-d} \|f\|_{1} \leq M(f) \leq \|f\|_{1}
\end{equation}
for each $f \in \mathbb{C}[x]$ of degree $d \geq 1$. Hence there are only finitely many polynomials $f \in \mathbb{Z}[x]$ of degree at most $D$ and $M(f) \leq T$, irrespective of how large $D$ and $T$ are. This shows that the minima in the definitions of $M(K)$ and $M(\Oseen_K)$ exist. 

Now recall that if $\alpha\in \Oseen_K$ with $K=\IQ(\alpha)$ then $f_{\alpha,\mathbb{Z}}$ is monic and $\mahdelta{f_{\alpha,\mathbb{Z}}}=\Delta_K [\Oseen_K:\IZ[\alpha]]^2$
where the positive integer $[\Oseen_K:\IZ[\alpha]]$ is the index of the order $\IZ[\alpha]$ in the maximal order $\Oseen_K$.
Therefore Mahler's inequality (Theorem \ref{thm:mahleroriginal})  implies
\begin{equation} \label{eqn:mahlerok}
M(\Oseen_K) \geq d^{-\frac{d}{2d-2}}|\Delta_K|^{\frac{1}{2d-2}}.
\end{equation}
Silverman \cite{silverman} proved this bound holds even for $M(K)$, and it is known\footnote{If $p<q<2p$ are two primes and $K=\IQ((p/q)^{1/d})$ then $M(K)\leq q
< (2pq)^{1/2}\leq 2^{1/2}|\Delta_K|^{1/(2d-2)}$.} (cf. \cite{ruppert}) that the exponent $\frac{1}{2d-2}$ on $|\Delta_K|$ is sharp for $M(K)$ and every $d\geq 2$.
 
Ruppert \cite{ruppert} showed that $M(K) \ll_d |\Delta_K|^{\frac{1}{2d-2}}$ when $d=2$ and asked \cite[Question 1]{ruppert} whether this remains true when $d> 2$. This was answered in the negative by Vaaler \& Widmer \cite{wout_smol_gens} for composite $d$,  and by Dubickas \cite{dubickas} for prime $d\geq 3$. Ruppert \cite[Question 2]{ruppert} asked also the analogous
{question with exponent $1/2$ instead of $1/(2d-2)$. This question is still open but has been affirmatively answered for not totally complex number fields, and, 
conditionally under GRH,  also for general number fields \cite{a_note}.

Less is known for $M(\Oseen_K)$.
The best  general upper bound is
\begin{equation*}
    M(\Oseen_K) \leq |\Delta_K|
\end{equation*}
which follows easily from Minkowski's convex body theorem (cf. \cite[Lemma 7.1]{pazuki}).
The question of whether the exponent $1/(2d-2)$ in (\ref{eqn:mahlerok}) is sharp seems more delicate than its counterpart for $M(K)$. 
The cubic case was affirmatively answered by Eldredge \& Petersen \cite[Theorem 1.1]{eldandpet}.
Motivated by a different question Jones published a result \cite[Theorem 1.5 (3)]{Jones} that gives an affirmative answer for all $d\in \{3,4,5,7,9\}$.
He shows that if $(d,w)\in \{(3,1), (4,2), (5,1), (7,1), (9,2)\}$ then there are infinitely many primes $t$ such that 
\begin{equation*}
P_{t}(x)=x^d-16d(dt+w)x^{d-1}+dt+w\in \IZ[x]
\end{equation*}
is irreducible and monogenic, i.e., $\Oseen_K=\IZ[\alpha]$ where $K=\IQ(\alpha)$ and $\alpha$ is any root of $P_{t}(x)$.
Hence, $\Delta_{P_{t}}=\Delta_K$, and Swan's discriminant formula shows that $|\Delta_{P_{t}}|\gg t^{2d-2}$ while
(\ref{ineq:L1norm}) implies that $M(\Oseen_K)\leq M(P_{t})\ll t$.

However, just as for  $M(K)$, the exponent $2d-2$ is certainly not always sharp when $d>2$. 
Eldredge \& Petersen \cite[Theorem 1.2]{eldandpet} proved for cubic, and Dubickas\footnote{Dubickas proves the sharper inequality $(1-\epsilon)|\Delta_K|^{\frac{1}{d}} < M(\Oseen_K) < |\Delta_K|^{\frac{1}{d}}$ for arbitray $\epsilon>0$.} \cite[Theorem 2]{dubickas} for arbitrary degrees $d \geq 2$, that there are infinitely many number fields $K$ of degree $d$ such that
\begin{equation*}
\frac{1}{30}|\Delta_K|^{\frac{1}{d}} < M(\Oseen_K) < \frac{4}{3}|\Delta_K|^{\frac{1}{d}}.
\end{equation*}
These fields $K$ are very special and are all of the form  $\IQ(p^{1/d})$ for certain primes $p$. 
It would be interesting to find improved lower bounds that apply to more general families of number fields.
Ruppert \cite{ruppert}  observed that for all imaginary quadratic fields
\begin{equation} \label{eqn:imaginarybound}
    M(\Oseen_K) \geq \frac{1}{4}|\Delta_K|.
\end{equation}
The bound (\ref{eqn:imaginarybound}) suggests that  such improvements of Mahler's exponent $1/(2d-2)$ 
might hold for the  family of totally complex fields of any fixed degree $d\geq 2$. 

Theorem \ref{thm:mainresultformahler} with $m=d/2$ and $r=1$ applied to  the minimal polynomial of  integral generators 
yields a generalisation of  Ruppert's observation  (\ref{eqn:imaginarybound}) to totally complex number fields
of arbitrary degree $d$.
\begin{corollary} \label{cor:numberfieldmahler}
Let $K$ be a totally complex number field of degree $d \geq 2$. Then
\begin{equation*}
    M(\Oseen_K) \geq 2^{\frac{d(1-d)}{2d-3}}|\Delta_K|^{\frac{1}{2d-3}}.
\end{equation*}
\end{corollary}
From number fields $K$ back to polynomials $f\in \IC[x]$, here is another application of Theorem \ref{thm:mainresultformahler}. 
Mahler \cite[Corollary]{mahler_1964}  combined  Theorem \ref{thm:mahleroriginal} and (\ref{ineq:L1norm}) to get 
\begin{equation}\label{ineq:L1disc}
|\mahdelta{\polynomial}|\leq d^d\|f\|_{1}^{2d-2}.
\end{equation}
For polynomials $f$ as in Theorem \ref{thm:mainresultformahler} we get a better bound whenever $\|f\|_1$ is sufficiently large.
\begin{corollary} \label{cor:L1normdisc}
For polynomials $f\in \IC[x]$ satisfying the conditions in Theorem \ref{thm:mainresultformahler} we have
\begin{equation*}
|\mahdelta{\polynomial}| \leq (2r)^{d(d-1)}\|f\|_1^{{2d-3}}.
\end{equation*}
\end{corollary}
Note that if $K$ is a not totally real number field of degree $d$,
and  $f=f_{\alpha,\mathbb{Z}}$ 
is as in Proposition \ref{thm:kcomplementtomahlerresult}, then  
\begin{equation*}
\|f\|_1^{{2d-3}}\ll_K |\mahdelta{\polynomial}|.
\end{equation*}

\section{Discrete logarithmic energy for point configurations}
Theorem \ref{thm:mainresultformahler} is proved by considering the complex roots of $f$ as points in $\IR^2$, and then bounding the logarithmic energy of this point configuration in $\IR^2$
from below in terms of the number of points and the product of their Euclidean norms (ignoring those points inside the unit disc).
Our simple argument is agnostic to the dimension and thus works for point configurations in arbitrary dimensions.

Let $k,d \in \mathbb{N}$, $d\geq 2$, let  $\altvectori{1}, \ldots, \altvectori{d}$ 
be points in $\mathbb{R}^k$, and set  $\vectorlist = (\altvectori{1}, \ldots, \altvectori{d})$. 
We define
\begin{equation*}
    \vecmahler{\vectorlist} := \prod_{i=1}^{d}\max\{1, |\altvectori{i}|\},
\end{equation*}
and
\begin{equation*}
    \absvecdelta{\vectorlist} := \prod_{1 \leq i < j \leq d}|\altvectori{i} - \altvectori{j}|^{2},
\end{equation*}
where $|\cdot|$ on the right hand-side denotes the Euclidean norm on  $\mathbb{R}^k$.

Note that $-\log|\Delta_{\vectorlist}|$ is the discrete logarithmic energy $ E_{\log}(\vectorlist)$ of the $d$-point configuration $\vectorlist$ in $\mathbb{R}^k$ (cf. \cite{grabner}),
and thus our goal is to bound $E_{\log}(\vectorlist)$ from below in terms of $\vecmahler{\vectorlist}$ and $d$. 
Mahler's inequality shows that $E_{\log}(\vectorlist)$ for a $d$-point configuration of points $\vectorlist$ on the closed unit disc in $\mathbb{R}^2$ is at least $-d\log d$,
and that this value is attained if and only if the points are equidistributed on the unit circle. But  finding the minimal logarithmic energy $E_0(d;k)$
for a  $d$-point configuration in the closed unit ball in $\mathbb{R}^k$ for higher dimensions $k\geq 3$
is a difficult open problem
about which we have nothing to say.

Let $\radius \geq 1$ be a real number, and let $m \in \{0,1,\ldots,d\}$ be the number of points $\altvectori{i}$ in $\vectorlist$ with Euclidean length strictly greater than $\radius$.
We prove the following generalisation of Theorem \ref{thm:mainresultformahler}.
\begin{theorem} \label{thm:mainresultforvecs}
Suppose $m$ is even and that the points in $\vectorlist$ with Euclidean length greater than $\radius$ occur in pairs of equal length. Then
\begin{equation} \label{eqn:vecresult}
    \vecmahler{\vectorlist}^{2d-3} \geq (2r)^{d(1-d)}|\mahdelta{\vectorlist}|.
\end{equation}
Moreover, the exponent $2d-3$ is sharp.
\end{theorem}

The case $m=0$ corresponds to an arrangement of $d$ points in a closed ball of radius $\radius$ in $\mathbb{R}^k$. In this case our theorem only yields the trivial lower bound 
$d(1-d)\cdot \log(2r)$ for the logarithmic energy. On the other hand it is clear from the proof that good lower bounds for $E_0(d-m;k)$ can be used to refine the constant $(2r)^{d(1-d)}$ in our theorem, at least if $d\geq m+2$. However, in this work we are only concerned with the exponent of the discriminant and for simplicity we have decided to record only the simplest possible explicit constant which is $(2r)^{d(1-d)}$.

\begin{proof}
Without loss of generality, we can relabel the points $\altvectori{1}, \ldots, \altvectori{d}$  in $\vectorlist$ in order of decreasing Euclidean length:
\begin{equation*}
    \vectori{1}, \vectori{1}', \ldots, \vectori{\frac{m}{2}}, \vectori{\frac{m}{2}}', \vectori{\frac{m}{2}+1}, \vectori{\frac{m}{2}+2}, \ldots, \vectori{d-\frac{m}{2}},
\end{equation*} where $|\vectori{i}| = |\vectori{i}'|$ ($i=1,\ldots,\frac{m}{2}$). Note that the points of Euclidean length strictly greater than $\radius$ are exactly the $m$ points $\vectori{1}, \vectori{1}', \ldots, \vectori{\frac{m}{2}}, \vectori{\frac{m}{2}}'$, so
\begin{equation*}
    |\vectori{1}| = |\vectori{1}'| \geq |\vectori{2}| = |\vectori{2}'| \geq \ldots \geq |\vectori{\frac{m}{2}}| = |\vectori{\frac{m}{2}}'| > \radius \geq |\vectori{\frac{m}{2}+1}| \geq \ldots \geq |\vectori{d-\frac{m}{2}}|.
\end{equation*}

If any two points are equal then $\absvecdelta{\vectorlist}=0$, and (\ref{eqn:vecresult}) is trivially true.
Hence, we can assume all of these points are distinct.

Recall that $\vectorlist = (\altvectori{1}, \ldots, \altvectori{d}) = (\vectori{1}, \vectori{1}', \ldots, \vectori{\frac{m}{2}}, \vectori{\frac{m}{2}}', \vectori{\frac{m}{2}+1}, \ldots, \vectori{d-\frac{m}{2}})$, so by definition of the discriminant we have
\begin{alignat}1 \label{eqn:big_product}
\nonumber \absvecdelta{\vectorlist} & = \prod_{1 \leq i < j \leq d}|\altvectori{i} - \altvectori{j}|^{2} \\
 & = \prod_{i=1}^{\frac{m}{2}}\left(|\vectori{i} - \vectori{i}'|^{2} \prod_{\smallervector \in \smallervectorsi{i}}|\vectori{i} - \smallervector|^{2}|\vectori{i}' - \smallervector|^{2}\right) \prod_{\frac{m}{2}+1 \leq i < j \leq d-\frac{m}{2}}|\vectori{i}-\vectori{j}|^{2},
\end{alignat}
where the set $\smallervectorsi{i} := \{\vectori{i+1}, \vectori{i+2}, \ldots, \vectori{d-\frac{m}{2}}, \vectori{i+1}', \vectori{i+2}', \ldots, \vectori{\frac{m}{2}}'\}$.

Let us start by evaluating the right-most product here, which we denote by $\littleconstant$ for convenience\footnote{Of course the sharp upper bound here is 
$\littleconstant\leq r^{(d-m)(d-m-1)}e^{-E_0(d-m;k)}$ where $E_0(d-m;k)$ denotes the minimal logarithmic energy of a $(d-m)$-configuration in the unit ball in $\IR^k$.}. We have $|\vectori{i}|, |\vectori{j}| \leq \radius$ whenever $\frac{m}{2}+1 \leq i < j \leq d-\frac{m}{2}$, so by the triangle inequality for all terms in this product,
\begin{equation*}
    |\vectori{i}-\vectori{j}|^2 \leq (|\vectori{i}| + |\vectori{j}|)^2 \leq (2\radius)^2.
\end{equation*}
Hence,
\begin{equation*}
\begin{split}
\littleconstant & \leq \prod_{m+1 \leq i < j \leq n}(2\radius)^2 = (2r)^{2{d-m \choose 2}} = (2r)^{(d-m)(d-m-1)}.
\end{split}
\end{equation*}
Now let us evaluate the left-most product in (\ref{eqn:big_product}). Noting that $|\vectori{i}| = |\vectori{i}'|$  ($i=1,\ldots,\frac{m}{2}$), we have for all $i=1,\ldots,\frac{m}{2}$,
\begin{equation*}
|\vectori{i} - \vectori{i}'|^{2} \leq (|\vectori{i}| + |\vectori{i}'|)^2 = (2|\vectori{i}|)^2 = 2^2|\vectori{i}|^2.
\end{equation*}
Likewise, noting that $|\smallervector| \leq |\vectori{i}| = |\vectori{i}'|$ for all $\smallervector \in \smallervectorsi{i}$ ($i=1,\ldots,\frac{m}{2}$), we have for all $\smallervector \in \smallervectorsi{i}$ ($i=1,\ldots,\frac{m}{2}$),
\begin{equation*} 
|\vectori{i} - \smallervector|^{2} \leq 2^2|\vectori{i}|^{2},
\end{equation*}
and
\begin{equation*} 
|\vectori{i}' - \smallervector|^{2} \leq 2^2|\vectori{i}|^{2}.
\end{equation*}

Combining this all together, we get
\begin{equation*}
\begin{split}
\absvecdelta{\vectorlist} & \leq \littleconstant \cdot \prod_{i=1}^{\frac{m}{2}}\left(2^2|\vectori{i}|^{2} \prod_{\smallervector\in\smallervectorsi{i}}2^2|\vectori{i}|^{2} \cdot 2^2|\vectori{i}|^{2}\right) 
   = \littleconstant \cdot \prod_{i=1}^{\frac{m}{2}}\left(2^2|\vectori{i}|^{2} \cdot (2^{4}|\vectori{i}|^{4})^{|\smallervectorsi{i}|}\right).
\end{split}
\end{equation*}
Now consider the number of elements in $\smallervectorsi{i}$. This is straightforward to calculate:
\begin{equation*}
    |\smallervectorsi{i}| = (d-\frac{m}{2}-i) + (\frac{m}{2}-i) = d - 2i.
\end{equation*}
So 
\begin{equation*}
\begin{split}
\absvecdelta{\vectorlist} & \leq \littleconstant \cdot \prod_{i=1}^{\frac{m}{2}}\left(2^2|\vectori{i}|^{2} \cdot (2^{4}|\vectori{i}|^{4})^{d-2i}\right) \\
 & = \littleconstant \cdot 2^{m} \cdot 2^{2dm} \cdot \prod_{i=1}^{\frac{m}{2}}\left(2^{-8i}|\vectori{i}|^{4d-8i+2}\right) \\
 & = \littleconstant \cdot 2^{m(2d+1)} \cdot \prod_{i=1}^{\frac{m}{2}}\left(2^{-8i}\right)\prod_{i=1}^{\frac{m}{2}}\left(|\vectori{i}|^{4d-6-8i+8}\right) \\
 & = \littleconstant \cdot 2^{m(2d-m-1)} \prod_{i=1}^{\frac{m}{2}}\left(|\vectori{i}|^{4d-6}\right)\prod_{i=1}^{\frac{m}{2}}\left(|\vectori{i}|^{-8i+8}\right).
\end{split}
\end{equation*}
Let
\begin{equation*}
\vecconstant{d,m} := \littleconstant \cdot 2^{m(2d-m-1)}.
\end{equation*} 
Then we see that
\begin{equation*}
\begin{split}
\vecconstant{d,m} & \leq (2r)^{d^{2}+m^{2}-2dm-d+m} \cdot 2^{m(2d-m-1)}\\
 & = (2r)^{d(d-1)+m(m+1-2d)} \cdot 2^{m(2d-m-1)} \\
 & \leq (2r)^{d(d-1)} \cdot 2^{m(m+1-2d+2d-m-1)} = (2r)^{d(d-1)},
\end{split}
\end{equation*}
where we used that $r\geq 1$ and $m+1-2d \leq d+1-2d \leq 0$.
Hence,
\begin{equation*}
\begin{split}
\absvecdelta{\vectorlist} & \leq (2r)^{d(d-1)} \prod_{i=1}^{\frac{m}{2}}\left(|\vectori{i}|^{4d-6}\right)\prod_{i=1}^{\frac{m}{2}}\left(|\vectori{i}|^{-8i+8}\right) \\
 & = (2r)^{d(d-1)} \prod_{i=1}^{\frac{m}{2}}\left(|\vectori{i}|^{2}\right)^{2d-3}\prod_{i=1}^{\frac{m}{2}}\left(|\vectori{i}|^{-8i+8}\right) \\
 & \leq (2r)^{d(d-1)} \vecmahler{\vectorlist}^{2d-3}\prod_{i=1}^{\frac{m}{2}}\left(|\vectori{i}|^{-8i+8}\right).
\end{split}
\end{equation*}

Consider the remaining product here. We have that $|\vectori{i}| > r \geq 1$ for all $i = 1, \ldots, \frac{m}{2}$; and $-8i+8 \leq 0$ also for all $i = 1, \ldots, \frac{m}{2}$. Hence, $0 < |\vectori{i}|^{-8i+8} \leq 1$ for all $i = 1, \ldots, \frac{m}{2}$, and thus $0 < \prod_{i=1}^{\frac{m}{2}}|\vectori{i}|^{-8i+8} \leq 1$. We conclude that
\begin{equation*}
\absvecdelta{\vectorlist} \leq (2r)^{d(d-1)}\vecmahler{\vectorlist}^{2d-3}
\end{equation*}
which  proves  inequality (\ref{eqn:vecresult}).

Finally, to see that the exponent $2d-3$ is sharp it suffices to consider $d-2$ distinct fixed points on the $(k-1)$-sphere of radius $\radius$ centred at the origin and taking the remaining two points $\vectori{1}$ and $\vectori{1}'=-\vectori{1}$ with arbitrarily large Euclidean distance. This completes the proof of Theorem \ref{thm:mainresultforvecs}.
\end{proof}

\section{Upper bounds for the Mahler measure of minimal polynomials of integral generators}

In this section we prove Proposition \ref{thm:kcomplementtomahlerresult}.
Let $K$ be a number field of degree $d$. Let 
$\sigma_i:K \rightarrow \mathbb{R}$ for $1 \leq i \leq r$ be the $r$ real embeddings, and let
 $\sigma_{i}, \sigma_{i+s}:K \rightarrow \mathbb{C}$  for $r+1 \leq i \leq r+s$ be the $s$ pairs of complex conjugate embeddings, so that $d=r+2s$.
Let
\begin{equation*}
\begin{split}
\boldsymbol{\sigma}:\ & K \rightarrow \mathbb{R}^r \times \mathbb{C}^s \\
& \alpha \mapsto (\sigma_1(\alpha), \ldots, \sigma_r(\alpha), \sigma_{r+1}(\alpha), \ldots, \sigma_{r+s}(\alpha))
\end{split}
\end{equation*}
be the Minkowski embedding of $K$, and set $\Lambda = \boldsymbol{\sigma}O_K \subseteq \mathbb{R}^r \times \mathbb{C}^s$. Identifying $\mathbb{C} \cong R^2$ via the isomorphism $\alpha \mapsto (\text{Re}(\alpha),\text{Im}(\alpha))$ turns $\Lambda$ into a lattice in $\mathbb{R}^d$. 
For the convenience of the reader we recall Proposition \ref{thm:kcomplementtomahlerresult}.
\begin{proposition} \label{thm:kcomplementtomahlerresultver2}
Let $K$ be a number field of degree $d$ with a non-real embedding (i.e. $s \geq 1)$. Then there exists $c_K > 0$ depending only on $K$ such that there are infinitely many $\alpha \in \Oseen_K$ with $K = \mathbb{Q}(\alpha)$,
\begin{equation*}
    M(f_{\alpha,\mathbb{Z}})^{2d-3} \leq c_K|\Delta_{f_{\alpha,\mathbb{Z}}}|,
\end{equation*}
and all roots of $f_{\alpha,\mathbb{Z}}$, except one pair of complex conjugate roots, have modulus at most $c_K$.
\end{proposition}

\begin{proof}
For real $c, T\geq 1$ we define the subset $S_{c,T}$ of $\mathbb{R}^r \times \mathbb{C}^s$ as the set of points ${\bf x} = (x_1, \ldots, x_{r+s}) \in \mathbb{R}^r \times \mathbb{C}^s$ such that
\begin{equation*}
\begin{cases}
        (2i-2)c < x_i < (2i-1)c \quad &\text{for} \ i = 1, \ldots, r \\
        (2i-2)c < \text{Re}(x_i) < (2i-1)c \quad &\text{for} \ i = r+1, \ldots, r+s-1 \\
        c < \text{Im}(x_i) < 2c \quad &\text{for} \ i = r+1, \ldots, r+s-1 \\
        T <  \text{Re}(x_{r+s}) < T+c \\
        3c+T <  \text{Im}(x_{r+s}) < T+4c. \\
     \end{cases}
\end{equation*}
If we identify $\mathbb{R}^r \times \mathbb{C}^s$ with $\mathbb{R}^d$ like before, then $S_{c,T}$ is a box with sides parallel to the axes, each of length $c$. Therefore, there exists $c_{\Lambda}$ such that $S_{c,T}$ contains a point of the lattice $\Lambda$ whenever $c \geq c_{\Lambda}$.

Furthermore, for all ${\bf x} = (x_1, \ldots, x_{r+s}) \in S_{c,T}$,
\begin{equation*}
    |x_i - x_j| > c \quad (1 \leq i < j \leq r+s).
\end{equation*}
This implies that for $\alpha \in K$ and $\boldsymbol{\sigma}(\alpha) \in S_{c,T}$,
\begin{equation} \label{greater_than_c}
    |\sigma_i(\alpha) - \sigma_j(\alpha)| > c \quad (1 \leq i < j \leq r+s).
\end{equation}
If $r+s+1 \leq i < j \leq r+2s$, then $|\sigma_i(\alpha) - \sigma_j(\alpha)| = |\overline{\sigma_{i-s}(\alpha) - \sigma_{j-s}(\alpha)}| = |\sigma_{i-s}(\alpha) - \sigma_{j-s}(\alpha)|> c$, 
by (\ref{greater_than_c}). If $1 \leq i \leq r+s$ and $j \geq r+s+1$ then
\begin{equation*}
|\sigma_i(\alpha) - \sigma_j(\alpha)|  \geq |\text{Im}(\sigma_i(\alpha)) - \text{Im}(\sigma_j(\alpha))| > |0 - (-c)|  = c.
\end{equation*}
We have shown that if $\alpha \in K$ and $\boldsymbol{\sigma}(\alpha) \in S_{c,T}$ then
\begin{equation} \label{sigma_ij_big_result}
    |\sigma_i(\alpha) - \sigma_j({\alpha})| > c
\end{equation}
whenever $1 \leq i < j \leq d$.

Take $c=c_{\Lambda} \geq 1$ as before. Then there exists $\alpha \in \Oseen_K$ such that $\boldsymbol{\sigma}(\alpha) \in \Lambda \cap S_{c,T}$. Let $f_{\alpha,\mathbb{Z}} \in \mathbb{Z}[x]$ be the minimal polynomial of $\alpha$.
Now (\ref{sigma_ij_big_result}) implies that $\sigma_i(\alpha)\neq \sigma_j(\alpha)$ for all  $1\leq i<j\leq d$. Hence, we conclude $\mathbb{Q}(\alpha) = K$.
Finally, using (\ref{sigma_ij_big_result}) we get
\begin{equation*}
\begin{split}
|\Delta_{f_{\alpha,\mathbb{Z}}}| &= \prod_{i \neq j}|\sigma_i(\alpha)-\sigma_j(\alpha)| \\
 &= \prod_{\substack{i \neq j \\ i,j \neq r+s, r+2s}}|\sigma_i(\alpha)-\sigma_j(\alpha)| \cdot \prod_{i \neq r+s}|\sigma_i(\alpha)-\sigma_{r+s}(\alpha)|^2 \cdot \prod_{\substack{i \neq r+2s \\ i \neq r+s}}|\sigma_i(\alpha)-\sigma_{r+2s}(\alpha)|^2 \\
 &\geq c^{(d-3)(d-2)} \cdot T^{2(d-1)} \cdot T^{2(d-2)} \\
 &=c^{(d-3)(d-2)} \cdot T^{2(2d-3)}.
\end{split}
\end{equation*}
On the other hand, using $|z| \leq |\text{Re}(z)| + |\text{Im}(z)|$ for $z \in \mathbb{C}$,
\begin{alignat}1\label{eqn:modulusproperty}
\nonumber M(f_{\alpha,\mathbb{Z}}) &= \prod_{i=1}^{d}\text{max}\{1,|\sigma_i(\alpha)|\} \\
\nonumber  &\leq (\left(2(r+s-1)+1)c\right)^{d-2}(2T+5c)^2 \\
 &\leq (2dc)^{d-2}(3T)^2 \\
\nonumber  &\leq (2dc)^dT^2,
\end{alignat}
provided $T \geq 5c$.
Hence,
\begin{equation*} 
\begin{split}
M(f_{\alpha,\mathbb{Z}})^{2d-3} &\leq (2dc)^{d(2d-3)}c^{-(d-3)(d-2)}|\Delta_{f_{\alpha,\mathbb{Z}}}| \\
 &\leq (2dc)^{d(2d-3)}|\Delta_{f_{\alpha,\mathbb{Z}}}| \\
 &= c_K|\Delta_{f_{\alpha,\mathbb{Z}}}|,
\end{split}
\end{equation*}
where $c_K:=(2dc)^{d(2d-3)}$ depends only on $K$, as $c_{\Lambda}$ depends only on $\boldsymbol{\sigma}\Oseen_K = \Lambda$.

Choosing a sequence of $T$'s, say $T_1,T_2,T_3,\ldots$ with $T_{i+1} > T_i + c$, yields a sequence of disjoint boxes $S_{c,T_i}$, each box containing an admissible $\alpha_i \in \Oseen_K$ with $K=\mathbb{Q}(\alpha_i)$ and $M(f_{\alpha_i,\mathbb{Z}})^{2d-3} \leq c_K|\Delta(f_{\alpha_i,\mathbb{Z}})|$. Moreover, as observed in (\ref{eqn:modulusproperty}), all roots of $f_{\alpha_i,\mathbb{Z}}$ have modulus at most $2dc \leq c_K$ except one pair of complex conjugate roots. This proves the theorem.
\end{proof}

It is clear that the proof of Proposition \ref{thm:kcomplementtomahlerresultver2} can be adapted to show that for any given number field $K$ of degree $d\geq 2$
with at least one real embedding there exist infinitely many $\alpha \in \Oseen_K$ with $K = \mathbb{Q}(\alpha)$, 
\begin{equation*}
    M(f_{\alpha,\mathbb{Z}})^{2d-2} \leq c_K|\Delta_{f_{\alpha,\mathbb{Z}}}|,
\end{equation*}
and all roots of $f_{\alpha,\mathbb{Z}}$, except one,  have modulus at most $c_K$.

\end{document}